\documentclass[12pt,reqno]{amsart}

\usepackage{fullpage}
\usepackage{amssymb}  
\usepackage{latexsym}
\usepackage{url}

\theoremstyle{plain}
\newtheorem{theorem}{Theorem}

\newtheorem{lemma}[theorem]{Lemma}
\newtheorem{proposition}[theorem]{Proposition}
\newtheorem{corollary}[theorem]{Corollary}
\newtheorem{problem}[theorem]{Problem}
\theoremstyle{remark}

\newcommand{\thmlabel}[1]{\label{thm:#1}}   
\newcommand{\lemlabel}[1]{\label{lem:#1}}   
\newcommand{\corlabel}[1]{\label{cor:#1}}   
\newcommand{\prplabel}[1]{\label{prp:#1}}   
\newcommand{\prblabel}[1]{\label{prb:#1}}   
\newcommand{\eqnlabel}[1]{\label{eqn:#1}}   

\newcommand{\thmref}[1]{\ref{thm:#1}}   
\newcommand{\lemref}[1]{\ref{lem:#1}}   
\newcommand{\corref}[1]{\ref{cor:#1}}   
\newcommand{\prpref}[1]{\ref{prp:#1}}   
\newcommand{\prbref}[1]{\ref{prb:#1}}   
\newcommand{\eqnref}[1]{\eqref{eqn:#1}} 

\DeclareMathOperator{\Mlt}{Mlt}
\DeclareMathOperator{\Inn}{Inn}
\DeclareMathOperator{\Aut}{Aut}

\DeclareMathOperator{\Sym}{Sym}
\DeclareMathOperator{\soc}{soc}

\newcommand{\setof}[2]{\{#1\,|\,#2\}}   
\newcommand{\normal}{\lhd}              
\newcommand{\inv}{^{-1}}                
\newcommand{\id}{\mathrm{id}}           
\newcommand{\ad}{\mathrm{ad}}           

\title{Solvability of commutative automorphic loops}

\author{Alexander Grishkov}
\author{Michael Kinyon}
\author{G\'{a}bor P. Nagy}

\address{Departamento de Matem\'{a}tica \\ Universidade de S\~{a}o Paulo\\
Caixa Postal 66281\\ S\~{a}o Paulo-SP, 05311-970, Brasil}
\email{\url{grishkov@ime.usp.br}}

\address{Department of Mathematics \\ University of Denver \\
2360 S Gaylord St \\ Denver, CO 80208, USA}
\email{\url{mkinyon@math.du.edu}}

\address{Bolyai Institute\\ University of Szeged\\
Aradi v\'{e}rtan\'{u}k tere 1, H-6720 Sze\-ged, Hungary}
\email{\url{nagyg@math.u-szeged.hu}}



\begin{document}

\begin{abstract}
We prove that every finite, commutative automorphic loop is solvable. We also prove that every finite, automorphic $2$-loop is solvable. The main idea of the proof is to associate a simple Lie algebra of characteristic $2$ to a hypothetical finite simple commutative automorphic loop. The ``crust of a thin sandwich'' theorem of Zel'manov and Kostrikin leads to a contradiction.
\end{abstract}

\maketitle

\section{Introduction}

A \emph{loop} $(Q,\cdot)$ is a set $Q$ with a binary operation $\cdot : Q\times Q\to Q$ such that (i) for each $a,b\in Q$, the equations $ax=b$ and $ya=b$ have unique solutions $x,y\in Q$, and (ii) there exists a neutral element $1\in Q$ such that $1x = x1 = x$ for all $x\in Q$. For $a\in Q$, the \emph{right translation} and \emph{left translation} by $a$ are the bijections $R_a :Q\to Q; x\mapsto xa$ and $L_a :Q\to Q;x\mapsto ax$. These generate the \emph{multiplication group} $\Mlt(Q) = \langle R_x, L_x\ |\ x\in Q\rangle$. The \emph{inner mapping group} is the subgroup stabilizing the neutral element, $\Inn(Q) = (\Mlt(Q))_1$.
A subloop $S$ of a loop $Q$ is \emph{normal} if it is the kernel of a homomorphism; this is equivalent both to $S$ being stabilized under the action of $\Inn(Q)$ and to $S$ being a block of $\Mlt(Q)$ containing $1$. A loop $Q$ is \emph{solvable} if it has a subnormal series $1\leq Q_0 \leq \cdots \leq Q_n = Q$, $Q_i \normal Q_{i+1}$, such that each factor loop $Q_{i+1}/Q_i$ is an abelian group. A loop is \emph{simple} if it has no nontrivial normal subloops. Basic references for loop theory are \cite{Bruck, Pflugfelder}.

A loop is \emph{automorphic} (or an \emph{A-loop}) if every inner mapping is an automorphism, that is, $\Inn(Q) \leq \Aut(Q)$. Automorphic loops were introduced by Bruck and Paige \cite{BP}. In recent years, a detailed structure theory has emerged for \emph{commutative} automorphic loops \cite{JKV1,JKV2,JKV3,Csorgo1}.
The outstanding open problem in the theory of automorphic loops is the following.

\begin{problem}
\prblabel{simple}
Does there exist a (finite) simple, nonassociative automorphic loop?
\end{problem}

\noindent It is known that there are no simple nonassociative automorphic loops of order less than $2500$ and no simple nonassociative commutative automorphic loops of order less than $2^{12}$ \cite{JKNV}. The main result of this paper shows that in the commutative case, Problem \prbref{simple} has a negative answer, and in fact, more than that.

\begin{theorem}
\thmlabel{main}
Every finite, commutative automorphic loop is solvable.
\end{theorem}

For a prime $p$, a finite loop $Q$ is said to be a $p$-\emph{loop} if $|Q| = p^m$ for some $m\geq 1$. A by-product of our proof of Theorem \thmref{main} is the following.

\begin{theorem}
\thmlabel{2loop}
Every automorphic $2$-loop is solvable.
\end{theorem}

Automorphic loops are power-associative, that is, each element generates a (cyclic) group \cite{BP}. In particular every element of an automorphic loop has a two-sided inverse. Unlike the situation for groups, in general power-associative loops, the property of being a $p$-loop is not equivalent to every element having order a power of $p$. This is, however, true for automorphic loops. For $p$ odd, this will be found in \cite{KKPV}. In the next section we will show that every finite automorphic loop consisting of elements of $2$-power order is a $2$-loop. We will address the (elementary) converse in the final section.

We also note that again unlike the group situation, automorphic $p$-loops are not necessarily nilpotent. Examples of nonnilpotent, commutative, automorphic $2$-loops can be found in \cite{JKV2}. Commutative, automorphic $p$-loops for $p > 2$ are indeed nilpotent, but there exist noncommutative, automorphic loops of order $p^3$ which are not \cite{JKV3}.

Just as for groups, if $Q$ is a loop with normal subloop $S$, then $Q$ is solvable if and only if both $S$ and $Q/S$ are solvable. Thus as one would expect, both Theorems \thmref{main} and \thmref{2loop} will follow from considering simple loops.

Theorem \thmref{main} itself reduces to considering finite simple commutative automorphic loops of exponent $2$ because of the following, which is a composite of Theorems 5.1 and 3.12 and Proposition 6.1 of \cite{JKV1}.
\begin{proposition}
\prplabel{composite}
Let $Q$ be a finite, commutative automorphic loop.
\begin{enumerate}
\item $Q \cong O(Q) \times E(Q)$ where $|O(Q)|$ is odd and $E(Q)$ is a $2$-loop.
\item If $|Q|$ is odd, then $Q$ is solvable.
\item If $Q$ is simple, then $Q$ has exponent $2$.
\end{enumerate}
\end{proposition}
\noindent In particular, suppose $Q$ is a minimal counterexample to Theorem \thmref{main}. If $Q$ were not simple, then $Q$ would have a normal subloop $S$ such that both $Q/S$ and $S$ are solvable. Since automorphic loops form a variety in the sense of universal algebra \cite{BP}, both $S$ and $Q/S$ are commutative and automorphic. But this contradicts the nonsolvability of $Q$. Therefore $Q$ is simple and, by Proposition \prpref{composite}, $Q$ is a $2$-loop of exponent $2$.

Similarly, suppose $Q$ is a minimal counterexample to Theorem \thmref{2loop}. Any subloop and any factor loop of a $2$-loop is a $2$-loop, and so by the same argument as in the preceding paragraph, it follows that $Q$ is simple. We will show in Theorem \thmref{reduce} that $Q$ must then be commutative and thus by Proposition \prpref{composite} have exponent $2$.

Thus both Theorems \thmref{main} and \thmref{2loop} will follow from showing that a finite simple commutative automorphic loop of exponent $2$ is a cyclic group of order $2$. This will be the main goal of the fourth section.

\section{Automorphic $2$-loops}

In this section, we prove that if every element of a finite automorphic loop has $2$-power order then the loop is a $2$-loop.

In a loop $Q$ with two-sided inverses, let $J: Q\to Q; x\mapsto x\inv$ denote the inversion map.

\begin{proposition}[\cite{JKNV}, Corollary 6.6]
\prplabel{aaip}
Every automorphic loop has the antiautomorphic inverse property, that is, the identity $(xy)\inv =y\inv x\inv$ holds for all $x,y$. Equivalently, $R_x^J = L_{x\inv}$ or $L_x^J = R_{x\inv}$ for all $x$.
\end{proposition}

\begin{corollary}
\corlabel{normalize}
If $Q$ is an automorphic loop, then $J \in N_{\Sym(Q)}(\Mlt(Q))\cap C_{\Sym(Q)}(\Inn(Q))$.
\end{corollary}

\begin{proof}
That $J$ normalizes $\Mlt(Q)$ follows because $\Mlt(Q)$ is generated by all $R_x$, $L_x$. That $J$ centralizes $\Inn(Q)$ follows because $\Inn(Q)\leq \Aut(Q)$, which is centralized by the antiautomorphism $J$.
\end{proof}

The following result was shown for commutative automorphic loops in \cite{JKV1}.

\begin{lemma} \label{lm:PaPbPa}
Let $Q$ be an automorphic loop. Define the map $P_x=R_x\inv L_{x\inv}=
R_x\inv R_{x}^J$. Then, for all $a,b \in Q$, we have $P_a P_b P_a = P_c$ with $c = b L_{a\inv}\inv R_a$. Moreover, $P_a^n =P_{a^n}$ holds for all integers $n$.
\end{lemma}
\begin{proof}
We have
\[
P_a P_b P_a = R_a\inv R_a^J R_b\inv R_b^J R_a\inv R_a^J
= g^{-1} g^J,
\]
where $g = R_b (R_a\inv)^J R_a$. Let $c = 1g = b R_a^J R_a = b L_{a\inv}\inv R_a$ and set $h = g R_c\inv$. Observe that $h\in \Inn(Q)$, so that $h\in \Aut(Q)$.
Thus
\[
P_a P_b P_a = R_c\inv h\inv J h R_c J = R_c\inv J R_c J =P_c\,,
\]
since $h$ centralizes $J$. This proves the first statement of the lemma. By
putting $a = b^{-1}$, we have $c = a$ and $P_a P_{a\inv} P_a = P_a$, which implies
$P_{a\inv}=P_a\inv$. Similarly, $b = 1$ implies $P_{a^2} = P_a^2$. Continuing this, one
obtains $P_{a^n} = P_a^n$ for all integers $n$.
\end{proof}

\begin{theorem}
\thmlabel{2power}
Let $Q$ be a finite automorphic loop. If every element of $Q$ has $2$-power
order then $Q$ is a $2$-loop.
\end{theorem}
\begin{proof}
Assume $Q$ is a minimal counterexample. As usual, $Q$ is simple and $\Inn(Q)$
is maximal in $\Mlt(Q)$. Let $G = \Mlt(Q)\langle J\rangle$, let $C$ denote the conjugacy class of $J$ in $G$ and let $X = \{ P_x \mid x\in Q\} \subset G$. As $J$ centralizes $\Inn(Q)$, we have $C = \{ J^{R_x} \mid x\in Q\}$. Moreover,
\[
J^{R_x} = R_x^{-1} J R_x = P_x J\,,
\]
thus, the map $g\mapsto gJ$ is a bijection between $C$ and $X$. Take any two
elements $J^a,J^b \in C$. As $b = h R_x a$ for some $h\in \Inn(Q)$ and $x\in Q$,
we have $J^b = J^{R_x a}$ and $J^b J^a = (J^{R_x} J)^a = P_x^a$. By
Lemma \ref{lm:PaPbPa}, the order of the permutation $P_x$ divides the order of
$x$, thus, the order of $J^b J^a$ is a power of $2$ for all $a,b \in G$ by
assumption. The Baer-Suzuki theorem (\cite{KS}, Thm. 6.7.6) implies that $C$
generates a nilpotent subgroup $H$ of $G$. As the Sylow $2$-group of $H$ is
normal in $H$, $H$ must be a $2$-group itself which is normal in $G$.

If $|H|>2$ then $H\cap \Mlt(Q)$ is a normal $2$-group of $\Mlt(Q)$, whose orbit
determines a nontrivial normal subloop of $2$-power order, a contradiction. If
$|H|=2$ then $H=\langle J \rangle$ and $J$ is central in $G$. This implies that $Q$
must be commutative of exponent $2$. By (\cite{JKV1}, Corollary 6.3), $Q$ has $2$-power
order, a contradiction.
\end{proof}

\section{The multiplication group of simple automorphic $2$-loops}

The starting point for our study of simple loops is the following important result, which is an immediate consequence of the characterization of normal subloops as blocks of the multiplication group (\cite{Albert}, Theorem 8).

\begin{proposition}
\prplabel{primitive}
A loop $Q$ is simple if and only if $\Mlt(Q)$ acts primitively on $Q$.
\end{proposition}

The simple loops under consideration here are all $2$-loops, and so we will use the classification of primitive groups of degree a power of $2$. This follows from the classification of nonabelian simple groups of prime power degree by Guralnick \cite{Guralnick} and is stated explicitly in \cite{GuSa}. For $p=2$, the result can be refined slightly using Zsigmondy's theorem \cite{Zsigm} as given in (3.3) of \cite{Guralnick}.

Recall that a primitive permutation group $G$ is of \emph{affine type} if it has an abelian
regular normal subgroup, which is necessarily elementary abelian of order $p^n$ for some
prime $p$. In this case $G$ is embedded in the affine group $AGL(n,p)$ with the socle being
the translation subgroup. The stabiliser of $0\in GF(p)^n$ is a subgroup of $GL(n,p)$ which
acts irreducibly on $GF(p)^n$.

\begin{proposition}[Guralnick and Saxl \cite{GuSa}]
\prplabel{GS}
Let $G$ be a primitive permutation group of degree $2^n$. Then either $G$ is of affine
type, or $G$ has a unique minimal normal subgroup $N=S\times \cdots \times S = S^t$, $t\ge
1$, $S$ is a nonabelian simple group, and one of the following holds:
\begin{enumerate}
\item[(i)] $S=A_m$, $m=2^e\geq 8$, $n=te$, and the point stabilizer in $N$ is $N_1 =
A_{m-1}\times \cdots \times A_{m-1}$, or
\item[(ii)] $S=PSL(2,q)$, $q = 2^e -1\geq 7$ is a Mersenne prime, $n=te$,
and the point stabilizer in $N$ is the direct product of maximal parabolic subgroups
each stabilizing a $1$-space.
\end{enumerate}
\end{proposition}

We will use the following result of Dr\'{a}pal (\cite{Drapal}, Theorem 5.1).

\begin{proposition}
\prplabel{Drapal}
Let $F$ be a finite field, $|F| \neq 3,4$, and let $Q$ be a loop with $\Mlt(Q) \leq P{\Gamma}L(2,F)$. Then $\Mlt(Q)\cong Q$ is a cyclic group.
\end{proposition}

We record one elementary fact about primitive groups.

\begin{lemma}
\lemlabel{primitive}
Let $G$ be a permutation group acting primitively on a set $\Omega$. Then for any $x\in \Omega$, $G_x$ acts fixed point free on $\Omega\setminus \{x\}$.
\end{lemma}

\begin{proof}
Assume that $y^g = y$ for all $g\in G_x$ and pick $h\in G$ such that $x^h=y$. Then $G_x\leq G_y = G_x^h$ and $h\in N_G(G_x)$. Since $G_x$ is maximal in $G$, we have $h\in G_x$ and so $y = x$.
\end{proof}

We also need the following consequence of (\cite{JKNV}, Lemma 4.1).

\begin{lemma}
\lemlabel{no4trans}
Let $Q$ be a loop and let $H \leq \Aut(Q)$. Then $H$ is not $4$-transitive on $Q\backslash \{1\}$.
\end{lemma}

We first eliminate all but the case of affine type in Proposition \prpref{GS}.

\begin{theorem}
\thmlabel{affine}
Let $Q$ be a simple automorphic $2$-loop. Then $\Mlt(Q)$ is a primitive group of affine type.
\end{theorem}

\begin{proof}
Suppose $\Mlt(Q)$ is not of affine type. By Proposition \prpref{GS}, $\Mlt(Q)$ contains a unique minimal normal subgroup $N=S^t$, $t\geq 1$ where $S$ a nonabelian simple group. The subgroup stabilizing $1\in Q$ is $N_1 = T_{(1)} \times \cdots \times T_{(t)}$ where each  $T_{(i)}$ is a maximal subgroup of $S$. In this case, we can identify $Q$ with the cartesian product $Q_{(1)} \times \cdots \times Q_{(t)}$, where $Q_{(i)}$ is the coset space $S/T_{(i)}$. We write the neutral element of $Q$ in the form $1=(1,\ldots,1)$.

Set $Q^* = \{ (x,1,\ldots,1) \mid x\in Q_{(1)} \}$. Since $S$ acts primitively on each
$Q_{(i)}$, Lemma \lemref{primitive} implies that each $T_{(i)}$ acts fixed-point freely on
$Q_{(i)}\setminus \{1\}$. Thus $Q^*$ is precisely the set of fixed points of the subgroup
$1\times T_{(2)} \times \cdots \times T_{(t)} \leq \Inn(Q)$. Since $\Inn(Q)\leq \Aut(Q)$,
$Q^*$ is a subloop of $Q$. Let $H$ denote the stabilizer subgroup of $Q^*$ in $\Mlt(Q)$,
and let $H^* \leq \Sym(Q^*)$ be the induced permutation group; $H^*\cong H/M$ where $M$
consists of those elements of $H$ acting trivially on $Q^*$.

By Dedekind's modular law, since $N\normal \Mlt(Q)$, 
$\hat{N} = N\cap H = S\times T_{(2)} \times \cdots \times T_{(t)}$
is a normal subgroup of $H$. Moreover, $M\cap \hat N = 1\times T_{(2)} \times \cdots
\times T_{(t)}$ acts trivially on $Q^*$, and, by $S \cong \hat N/(M\cap \hat N) \cong
M\hat N/M$, the induced action of $\hat N$ on $Q^*$ is permutation equivalent to the
action of $S$ on $Q_{(1)}$. Since $M\hat N$ is normal in $H$, the permutation group $H^*$
on $Q^*$ has $S\times 1\times \cdots \times 1 \cong S$ as a normal subgroup. Similarly, we
can show that $T_{(1)}\times 1 \times \cdots \times 1 \cong T_{(1)} =S_1$ consists of
automorphisms of the loop $Q^*$. After identifying the groups $S\times 1\times \cdots
\times 1$, $T_{(1)} \times 1\times \cdots \times 1$ with $S$ and $T_{(1)}$, respectively,
we have
\begin{equation}
\eqnlabel{H*}
\Mlt(Q^*) \leq H^* \leq N_{\Sym(Q^*)}(S)\qquad\text{and}\qquad T_{(1)}\leq H^*_1\leq
\Aut(Q^*).
\end{equation}

Now assume that case (i) of Proposition \prpref{GS} holds, that is, $S = A_m$, $T_{(1)} = A_{m-1}$ with $m = 2^e\geq 8$. Then $\Aut(Q^*)$ is $5$-transitive on $Q^*\setminus \{1\}$, which is impossible by Lemma \lemref{no4trans}.

Now assume that case (ii) of Proposition \prpref{GS} holds, that is, $S = PSL(2,q)$ with $q = 2^e - 1\geq 7$ a Mersenne prime, and each $T_{(i)}$ a maximal parabolic subgroup stabilizing a $1$-space. In this case, $N_{\Sym(Q^*)}(PSL(2,q)) = P\Gamma{}L(2,q)$, and so by \eqnref{H*}, $\Mlt(Q^*)\leq P{\Gamma}L(2,q)$. By Proposition \prpref{Drapal}, $Q^*$ is a cyclic group. This contradicts the assumption that $T_{(1)}\leq \Aut(Q^*)$ operates transitively on $Q^*\setminus \{1\}$.
\end{proof}

Now we show that like Theorem \thmref{main}, Theorem \thmref{2loop} reduces to considering commutative automorphic loops of exponent $2$.

\begin{theorem}
\thmlabel{reduce}
Every simple, automorphic $2$-loop is commutative of exponent $2$.
\end{theorem}

\begin{proof}
Let $Q$ be a simple, automorphic $2$-loop. By Theorem \thmref{affine}, $\Mlt(Q)$ is of affine type. Thus $U = \soc(\Mlt(Q))$ is a regular, normal, elementary abelian $2$-subgroup. We identify $U$ with a $GF(2)$-vector space and we identify $\Inn(Q)$ with an irreducible subgroup of $GL(U)$.
Since $U$ is characteristic in $\Mlt(Q)$, Corollary \corref{normalize} gives $J\in N_{\Sym(Q)}(U)$. Hence the group $U\langle J\rangle$ is a $2$-group, and so $1\neq Z(U\langle J\rangle) = C_U(J)$. Since $J$ centralizes $\Inn(Q)$, irreducibility of $\Inn(Q)$ gives $C_U(J) = U$ and so $J\upharpoonright U = \id_U$. Thus $J = \id_Q$ and so $Q$ has exponent $2$. Then $Q$ is commutative since $J$ is an antiautomorphism (Proposition \prpref{aaip}).
\end{proof}

\section{Automorphic loops and Lie algebras}

We can now prove the main results of the paper by eliminating the case of affine type in Proposition \prpref{GS}. In the proof, we construct a simple Lie algebra from a hypothetical simple commutative automorphic loop of exponent $2$. The ``crust of a thin sandwich'' theorem of Zel'manov and Kostrikin will lead to a contradiction.

Let $Q$ be a finite, simple, commutative, automorphic loop of exponent $2$. \emph{We assume from now on} that $Q$ is not associative, and we will work toward a contradiction. Again, $\Mlt(Q)$ is of affine type and we identify $U = \soc(\Mlt(Q))$ with a $GF(2)$-vector space, the operation of which we now write additively. Once again, we identify $\Inn(Q)$ with an irreducible subgroup of $GL(U)$. Each right translation $R_x$, $x\in Q$ can be factored as
$R_x = h_x u_x$ for a unique $h_x\in \Inn(Q)$ and a unique $u_x\in U$.

Set $R_{x,y} = R_x R_y R_{xy}\inv$ and note that $R_{x,y}\in \Inn(Q)$. Then
\[
R_{x,y} h_{xy} u_{xy} = R_{x,y} R_{xy} = R_x R_y
= h_x u_x h_y u_y = h_x h_y (u_x^{h_y} + u_y)\,.
\]
Therefore
\begin{equation}
\eqnlabel{factors}
R_{x,y} = h_x h_y h_{xy}^{-1}\qquad\text{and}\qquad u_{xy} = u_x^{h_y} + u_y\,.
\end{equation}
Now we also have a one-to-one correspondence between $U$ and the set $\setof{h_x}{x\in Q}$. Abusing notation a bit, we may thus index elements of the latter set by elements of $U$: $h_u = h_x$ where $R_x = h_x u$. This allows us to define an isomorphic copy of $Q$ on $U$ by
\begin{equation}
\eqnlabel{newop}
u\circ v = u^{h_v} + v\,.
\end{equation}
Denote the right translations in $(U,\circ)$ by $R^{\circ}_u : v \mapsto v\circ u$, and for $u,v\in U$, set $R^{\circ}_{u,v} = R^{\circ}_u R^{\circ}_v (R^{\circ}_{u\circ v})\inv$. For all $u,v,w\in U$
\begin{equation}
\eqnlabel{Uinn}
w R^{\circ}_{u,v}
= \{[(w^{h_u} + u)^{h_v} + v] + (u^{h_v} + v)\}^{h_{u\circ v}\inv}
= \{(w^{h_u h_v} + u^{h_v} + v + u^{h_v} + v\}^{h_{u\circ v}\inv}
= w^{h_u h_v h_{u\circ v}\inv}
\end{equation}

\begin{lemma}
\lemlabel{inn=inn}
$\Inn(U,\circ) = \Inn(Q) = \langle h_x\ |\ x\in Q\rangle$.
\end{lemma}

\begin{proof}(\emph{cf}. \cite{JKNV}, Lemma 6.1)
Set $H = \langle h_x\ |\ x\in Q\rangle$. Since each $h_x\in \Inn(Q)$, $H\leq \Inn(Q)$.
Because $Q$ is commutative, $\Inn(Q)$ is generated by the mappings $R_{x,y}$ \cite{BruckCont}. By \eqnref{factors}, we have $R_{x,y} = h_x h_y h_{xy}^{-1}$, and so $\Inn(Q) = \langle h_x h_y h_{xy}^{-1}\ |\ x,y\in Q\rangle\leq H$. Similarly, by \eqnref{Uinn}, $\Inn(U,\circ)\leq H$. But $\Inn(Q)$ and $\Inn(U)$ are isomorphic, and hence, by finiteness, equal.
\end{proof}

\begin{lemma}
\lemlabel{aut_char}
For all $u,v\in U$, $h_u h_v = h_v h_{u^{h_v}}$.
\end{lemma}

\begin{proof}
Since $Q$ is automorphic, we have for all $u,v,w\in U$,
\[
w^{h_v h_{u^{h_v}}} + u^{h_v} = w^{h_v}\circ u^{h_v}
= (w\circ u)^{h_v} = (w^{h_u} + u)^{h_v} = w^{h_u h_v} + u^{h_v}\,.
\]
The desired result follows immediately.
\end{proof}

Next we define a new binary operation on $U$ as follows:
\begin{equation}
\eqnlabel{lieop}
[u,v] = u + v + u\circ v
\end{equation}
for all $u,v\in U$. Evidently, $[\cdot,\cdot ]$ is commutative and $[u,u] = 0$ for all $u\in U$. In addition, $[\cdot,\cdot ]$ turns out to be $GF(2)$-bilinear.

\begin{proposition}[\cite{Wright}, Theorem 4]
\prplabel{Wright}
$(U,+,[\cdot, \cdot ])$ is a simple, nonassociative algebra over $GF(2)$.
\end{proposition}

For $u\in U$, let $\ad(u) : U\to U; v\mapsto [v,u]$ denote the right multiplication mapping in the algebra $(U,+,[\cdot, \cdot ])$. We use these notation conventions in anticipation of the following result.

\begin{lemma}
\lemlabel{lie_alg}
$(U,+,[\cdot, \cdot ])$ is a simple Lie algebra over $GF(2)$ satisfying
\begin{equation}
\eqnlabel{premedial}
\ad(u)\ad([u,w]) = 0
\end{equation}
for all $u,w\in U$.
\end{lemma}

\begin{proof}
We have $(v)\ad(u) = u + v + v^{h_u} + u = v(\id_U + h_u)$, that is,
\[
h_u = \id_U + \ad(u)
\]
for all $u\in U$. We now use Lemma \lemref{aut_char}. First we compute
\[
h_u h_v = (\id_U + \ad(u))(\id_U + \ad(v)) = \id_U + \ad(u) + \ad(v) + \ad(u) \ad(v)\,.
\]
Since $\ad(u+ (u)\ad(v)) = \ad(u) + \ad([u,v])$, we also have
\begin{align*}
h_v h_{u^{h_v}} &= (\id_U + \ad(v))(\id_U + \ad(u) + \ad([u,v])) \\
&= \id_U + \ad(v) + \ad(u) + \ad(v) \ad(u) + \ad([u,v]) + \ad(v) \ad([u,v])\,.
\end{align*}
Equating both expressions, we have
\[
\ad(u) \ad(v) = \ad(v) \ad(u) + \ad([u,v]) + \ad(v) \ad([u,v])\,,
\]
or equivalently,
\begin{equation}
\eqnlabel{J}
\ad(u) \ad(v) + \ad(v) \ad(u) + \ad([u,v]) = \ad(v) \ad([u,v])\,.
\end{equation}
Since the left side is invariant under switching the roles of $u$ and $v$, so is the right side, and thus we have
\begin{equation}
\eqnlabel{pretmp}
\ad(u) \ad([v,u]) = \ad(v) \ad([u,v])
\end{equation}
for all $u,v\in U$. Now we linearize \eqnref{pretmp} by replacing $v$ with $v + w$. We get
\begin{align*}
&\ad(u) \ad([v,u]) + \ad(u) \ad([w,u]) \\ &=
\ad(v) \ad([u,v]) + \ad(v) \ad([u,w]) + \ad(w) \ad([u,v]) + \ad(w) \ad([u,w])\,.
\end{align*}
Using \eqnref{pretmp} to cancel terms, we obtain
\[
\ad(v) \ad([u,w]) + \ad(w) \ad([u,v]) = 0\,.
\]
Set $v = u$ and use $[u,u] = 0$ to get
\[
\ad(u) \ad([u,w]) = 0\,.
\]
This establishes \eqnref{premedial}. Applying \eqnref{premedial} to \eqnref{J}, we have
\[
\ad(u) \ad(v) + \ad(v) \ad(u) + \ad([u,v]) = 0\,,
\]
which is precisely the Jacobi identity (in characteristic $2$). Therefore $(U,+,[\cdot,\cdot])$ is a Lie algebra. The simplicity was already mentioned in Proposition \prpref{Wright}.
\end{proof}

For the final step, we will need the well-known ``crust of a thin sandwich'' theorem of Zel'manov and Kostrikin \cite{ZK}.

\begin{proposition}
\prplabel{sandwich}
Let $\mathfrak{g}$ be a Lie ring generated by a finite collection of elements $a$ satisfying $\ad(a)^2 = 0$ and $\ad(a)\ad(x)\ad(a) = 0$ for all $x\in \mathfrak{g}$. Then $\mathfrak{g}$ is nilpotent.
\end{proposition}

\begin{lemma}
\lemlabel{sandwich2}
Let $\mathfrak{g}$ be a Lie ring satisfying \eqnref{premedial}. If $\mathfrak{g}$ is generated by finitely many elements of the form $[x,y]$, then $\mathfrak{g}$ is nilpotent.
\end{lemma}

\begin{proof}
We have
\begin{equation}
\eqnlabel{sandtmp}
\ad([x,y])^2 = \ad(x)\ad(y)\ad([x,y]) + \ad(y)\ad(x)\ad([y,x]) = 0\,,
\end{equation}
by \eqnref{premedial}. Also,
\[
\ad([x,y])\ad(z)\ad([x,y]) = \ad([x,y])^2 \ad(z) + \ad([x,y])\ad([z,[x,y]]) = 0\,,
\]
using \eqnref{sandtmp} and \eqnref{premedial}. The conditions of Proposition \prpref{sandwich} are satisfied, and so $\mathfrak{g}$ is nilpotent.
\end{proof}

Returning to our Lie algebra $(U,+,[\cdot,\cdot])$, we now obtain a contradiction as follows. Since $(U,+,[\cdot,\cdot])$ is simple, we have $[U,U] = U$. Thus $(U,+,[\cdot,\cdot])$ is generated by finitely many elements of the form $[x,y]$. By Lemmas \lemref{lie_alg} and \lemref{sandwich2}, $(U,+,[\cdot,\cdot])$ is nilpotent, a contradiction.

We have seen that a simple, commutative, automorphic loop of exponent $2$ cannot be nonassociative, and hence must be a cyclic group of order $2$. This completes the proofs of Theorem \thmref{main} (by Proposition \prpref{composite}) and of Theorem \thmref{2loop} (by Theorem \thmref{reduce}).

\section{Final remarks}

We note that the converse of Theorem \thmref{2power} is an immediate consequence of Theorem \thmref{2loop}: if $Q$ is an automorphic $2$-loop, then since $Q$ is solvable, it follows from the same argument as in group theory that $Q$ has a subloop $S$ of index $2$. We have $x^2\in S$ for all $x\in Q$, and then by an induction argument, $x^2$ must have $2$-power order. Thus so does $x$ and hence every element of $Q$ has $2$-power order.

A subloop $S$ of a loop $Q$ is \emph{characteristic} if it is invariant under $\Aut(Q)$. Clearly every characteristic subloop of an automorphic loop is normal. In fact, standard facts about characteristic subgroups of groups hold for characteristic subloops of automorphic loops with essentially identical proofs. For instance, a characteristic subloop $T$ of a normal subloop $S$ of an automorphic loop $Q$ is necessarily normal in $Q$.

The \emph{derived subloop} $Q'$ of a loop $Q$ is the smallest normal subloop of $Q$ such that $Q/Q'$ is an abelian group. The derived subloop is characteristic. It follows from the above remarks that if $Q$ is automorphic, then each higher derived subloop $Q''$, $Q'''$, \emph{etc.} is normal in $Q$. In particular, the derived series of a solvable automorphic loop is a \emph{normal} series $Q \unrhd Q' \unrhd Q'' \unrhd \cdots\unrhd Q^{(n)} = 1$.

In a similar vein, we note that a minimal normal subloop of a finite automorphic loop is a direct product of isomorphic simple automorphic loops. Indeed, the same argument that works for groups (\emph{e.g.}, \cite{Huppert}, p. 51) applies without change. Thus a minimal normal subloop of a finite solvable automorphic loop is an elementary abelian $p$-group. These remarks may prove helpful in settling one of the main remaining open problems in the theory of commutative automorphic loops:

\begin{problem}
\prblabel{Sylow}
Let $Q$ be a commutative automorphic loop. For each prime $p$, does $Q$ have a Sylow $p$-subloop? For each set of primes $\pi$, does $Q$ have a Hall $\pi$-subloop?
\end{problem}

\end{document}